\documentclass[11pt]{article}
\usepackage{latexsym,amsmath,amsthm,verbatim,ifthen,amssymb,graphicx,color}
\usepackage[all]{xy}

\SelectTips{cm}{}

\newtheorem{theorem}{Theorem}[section]

\newtheorem{corollary}[theorem]{Corollary}
 
 \newtheorem{proposition}[theorem]{Proposition}
 \theoremstyle{definition}
 \newtheorem{definition}[theorem]{Definition}
 \theoremstyle{remark}
 \newtheorem{remark}[theorem]{Remark}
 \newtheorem{example}[theorem]{Example}
 \numberwithin{equation}{subsection}

\usepackage{color}
\usepackage{qtree}
\def\trd{\textcolor{red}}

\newcommand{\bz}{\mathbb Z}
\newcommand{\bi}{\mathbb I}
\newcommand{\bx}{\mathbb X}
\newcommand{\by}{\mathbb Y}

\newcommand{\bv}{\mathbb V}
\newcommand{\bu}{\mathbb U}
\newcommand{\bk}{\mathbb K}
\newcommand{\bq}{\mathbb Q}

\newcommand{\bs}{\mathbb S}

\newcommand{\im}{\text{Im}\,}

 \newcommand{\lib }{\mathbb{L}}

  \newcommand{\basis}{\operatorname{{\mathcal B}}}

\newcommand{\catdga}{\operatorname{{\bf DGA}}}
\newcommand{\catdgl}{\operatorname{{\bf DGL}}}

\newcommand{\De}{\Delta}
\newcommand{\al}{\alpha}
\newcommand{\be}{\beta}

\begin{document}

\title{$A_\infty$-Persistence}

\author{Francisco Belch\'{\i} and Aniceto Murillo* \let\thefootnote\relax\footnote{This work has been supported by the \emph{Ministerio de Educaci\'on y
Ciencia}  grant MTM2010-18089 and by the Junta de
Andaluc\'\i a grant FQM-213.  \vskip
 1pt
 Key words and phrases: persistent homology, $A_\infty$-persistence, $A_\infty$-coalgebra, applied algebraic topology.}}

\maketitle

\begin{abstract}
We introduce and study $A_\infty$-persistence of a given homology filtration of topological spaces. This is a family, one  for each $n\ge 1$, of homological invariants which provide information not readily available by the 
(persistent)
Betti numbers of the given filtration. This may help to detect noise, not just in the simplicial structure of the filtration but in further geometrical properties in which the higher codiagonals of the $A_\infty$-structure are translated. Based in the classification of zigzag modules, a characterization of the $A_\infty$-persistence in terms of its associated barcode is given.
\end{abstract}

\section*{Introduction}

	Persistent Homology \cite{CZ05,ELZ02,Ghr08,zo}, a topological technique which has proved to be useful in areas such as  digital imaging \cite{RSW11}, sensor networks \cite{DSG07} or bioinformatics \cite{Car09}, allows to extract global structural information
	about data sets (specially high dimensional ones) which may contain noise.

	The strategy consists on building first a filtration of topological spaces describing the behavior of the original data set concerning either its evolution in time or different scales or thresholds of fineness with respects to certain inherent properties of the data. Then, the analysis of the homology sequence induced by that filtration, particularly of those homology classes which are ``resistent to vanish", might be interpreted in term of noise produced by the initial data.

However, plain homology, i.e., the Betti numbers of the filtration, may not be always enough to analyze  noise with respect to finer geometrical-homological properties of the data. For instance, Betti numbers cannot tell apart
	a torus from the wedge of spheres $\bs^1 \vee \bs^1 \vee \bs^2$ and the computation of the cup product is needed. But again,
	the Borromean rings and three unlinked rings have both trivial cohomology algebra and one needs to compute the set of Massey product of the resulting links to set them apart.

In this paper we present a way of overcoming these obstacles by introducing persistence of the standard
$A_\infty$-coalgebra structure associated to the homology of the original topological filtration. In general, see next section for precise definition and details, an $A_\infty$-coalgebra structure on a graded module $M$ is a sequence of  operations
$$ \De_n: M \longrightarrow M^{\otimes n}\qquad n\ge 1$$
satisfying certain coassociativity relations. Roughly speaking, each $\Delta_n$ makes $M$ a coassociative coalgebra up to the operations $\Delta_j$, $j\le n$. Given a topological space $X$, its homology $H_*(X)$ is always endowed with a standard $A_\infty$-coalgebra structure for which $\Delta_2$ is the induced map in homology by any approximation to the diagonal.

In this context, we define {\em $A_\infty$-persistence}, a family of homological invariants of the filtration, one for each $n\ge 1$,  to understand the resistance of any of these high diagonals to vanish along the filtration. After observing that $A_\infty$-persistence contains classical persistence as the first invariant of the family,  we see that, unlike in this classical setting,  homology classes may be born and die several times along the filtration with respect to any high diagonal $\Delta_n$, see Theorem \ref{nacermorir}. Nevertheless, using the classification of  finitely generated zigzag modules \cite{CdS08}, we are able to prove that $A_\infty$-persistence of the given filtration is characterized in term of barcodes, see Theorem \ref{principal}.

We point out that this paper contains the mathematical foundations of $A_\infty$-persistence over which the experimental and applicable continuation is to be built. The dualization of the presented results, which in particular contains a new approach to persistence of cup product  \cite{hebbol}, the specific algorithms needed to compute $A_\infty$-persistence of a given filtration, a germ of which can be found in \cite{BM-AR10} for $\Delta_2$ and $\Delta_3$, and their implementations in different applications (3D-image processing, data analysis...),  will be treated elsewhere.

The organization of this paper is as follows: in Section 1 we briefly present the algebraic tools to be used and the standard $A_\infty$-coalgebra structure to be considered in any given topological space. We also show, see Theorem \ref{otro}, that even under heavy constraints, spaces may have isomorphic cohomology algebras and different  $A_\infty$-coalgebra structures.  Section 2 is devoted to define $A_\infty$-persistence, obtain their first properties, and prove Theorem \ref{principal}, the description of
$A_\infty$-persistence of a given filtration. In Section 3 we prove Theorem \ref{nacermorir}, which shows the $A_\infty$-persistence particular behavior   by which homology classes may be born and die several times. Finally, as some of the results in the paper
lie heavily
in classical facts from rational homotopy theory, we include for completeness an appendix containing an overview of this theory and the results we use.

\section{$A_\infty$-coalgebras}
Here, we recall the basic facts we shall use from $A_\infty$-structures. Unless explicitly stated otherwise, we fix a ring $\bk$, often a field, as coefficient for any considered graded algebraic object, including the homology of topological spaces. We assume the reader is familiar with the usual objects in differential homological algebra.

		An {\em $A_\infty$-coalgebra structure} $(M, \{ \De_n \}_{n \geq 1} )$
		on a graded module $M$ consists of a
		sequence of morphisms
		$$ \De_n: M \longrightarrow M^{\otimes n}=\underbrace{M \otimes \ldots \otimes M}_{n} $$
		of degree $n-2$
		such that, for any $n \geq 1$,
		$$ \sum_{i=1}^n \sum_{k=0}^{n-i} (-1)^{i+k+ik}
		\left( 1^{\otimes n-i-k} \otimes \De_i \otimes 1^k \right)
		\,\De_{n-i+1} = 0. $$

In particular, $\De_1$ is a differential on $M$ and $ (\De_1 \otimes 1 + 1 \otimes \De_1) \De_2 = \De_2 \De_1$. Note that a differential graded coalgebra, DGC henceforth, $(M,\Delta)$ is an $A_\infty$-coalgebra for which $\De_n=0$ for $n\ge 3$.

\begin{remark}\label{primerremark} It is convenient to note that  $A_\infty$-coalgebra structures in the graded module $M$ are in one-to-one correspondence with differentials in the complete tensor algebra $\widehat T(s^{-1}M)=\Pi_{n\ge 1}T^n(s^{-1}M)$ with $T^n(s^{-1}M)=s^{-1}M^{\otimes n}$. Here, $s^{-1}M$ denotes the desuspension of $M$, that is, $(s^{-1}M)_p=M_{p+1}$. In fact,   a differential $d$  in this algebra is determined by its image  on $s^{-1}M$, which can be written as a sum $d=\sum_{n\ge 1}d_n$, with $d_n(s^{-1}M)\subset T^{ n}(s^{-1}M)$, for $n\ge 1$. Then, the operators $\{\Delta_n\}_{n\ge 1}$ and $\{d_n\}_{n\ge 1}$ uniquely determine each other via
$$
\begin{aligned}
\Delta_n&=-s^{\otimes n}\circ d_n\circ s^{-1}\colon M\to
M^{\otimes n},\\
d_{n}&=-(-1)^{\frac{n(n-1)}{2}}(s^{-1})^{\otimes n}\circ\Delta_n\circ s\colon s^{-1}M\to T^{ n}(s^{-1}M).
\end{aligned}\eqno{(8)}
$$
\end{remark}
	
A morphism of $A_\infty$-coalgebras $f\colon (M,\{\De_n\}_{n\ge 1})\to (N,\{\De'_n\}_{n\ge 1})$ is a family of linear maps
 $$
 f_{(k)}\colon M\longrightarrow N^{\otimes k},\qquad k\ge 1,
 $$
each of which is of degree $k-1$, and satisfying the following identities for each $i\ge 1$:
$$
\sum_{\begin{aligned}&{\scriptstyle p+q+k=i}\\&{\scriptstyle j=p+k+1}\end{aligned}} (1^{\otimes p}\otimes\Delta_q\otimes 1^{\otimes k})f_{(j)}=\sum_{k_1+\dots+k_\ell=i}(f_{(k_1)}\otimes\dots\otimes f_{(k_\ell)})\Delta_\ell.
$$
This is equivalent to the existence of a morphism of differential graded algebras (DGA henceforth), which we denote in the same way,
$$
f\colon (\widehat T(s^{-1}M),d)\longrightarrow (\widehat T(s^{-1}N),d'),
$$
being $d$ and $d'$ the differentials corresponding, via Remark \ref{primerremark}, to the\break $A_\infty$-coalgebra structures  on $M$ and $N$ respectively. Indeed, write $f_{|_{s^{-1}M}}=\sum_{k\ge 0} f_k$ with $f_k\colon s^{-1}M\to T^{k}(s^{-1}N)$ and consider the morphisms\break $f_{(k)}\colon M\to N^{\otimes k}$ induced by $f_k$  eliminating desuspensions. Then, it is straightforward to check that the equality $fd=d'f$ translates to  identities as above satisfied by the maps $\{f_{(k)}\}_{k\ge 0}$.

The classical {\em Perturbation Lemma} \cite{GLS91,HK91,Kad80},
	known nowadays as the {\em Homotopy Transfer Theorem} \cite{KS99, LV}, is a common tool to find $A_\infty$-coalgebra structures as deformations of DGC's. We recall this result here in the best adapted form for our purposes.

\begin{theorem}\label{htt}
Let
\begin{equation}\label{ecuacion1}
\xymatrix{ \ar@(ul,dl)@<-5.5ex>[]_\phi  & (M,d) \ar@<0.75ex>[r]^-p & (N,d) \ar@<0.75ex>[l]^-\iota }
\end{equation}
be a diagram of chain complexes in which $(M,d)$ is a DGC with diagonal $\Delta$. Assume that $p\iota={\rm id}_N$, $p \phi=\phi\iota=\phi^2=0$ and $\phi$ is a chain homotopy between ${\rm id}_M$ and $\iota p$, i.e., $\phi d+d\phi=\iota p-{\rm id}_M$. Then, there is an explicit $A_\infty$-coalgebra structure on $N$ and morphisms of $A_\infty$-coalgebras,
$$
\xymatrix{M \ar@<0.75ex>[r]^-P& N\ar@<0.75ex>[l]^-I },
$$
Such that $P_{(1)}=p$ and $I_{(1)}=\iota$. \end{theorem}
In what follows, a diagram like \ref{ecuacion1} satisfying the required properties will be called a {\em transfer diagram}. As we will only use functorial general properties of the transfered $A_\infty$-structure, we will not make it explicit and refer the reader to the given references.

Via this result, there is a standard and functorial way of endowing  the singular homology  of a given complex with a structure of $A_\infty$-coalgebra. For it, consider the DGC structure on the singular chains $C_*(X)$ given by the {\em Alexander-Whitney diagonal}. Then, it is easy to find a transfer diagram of the form
$$
		\xymatrix{ \ar@(ul,dl)@<-5.5ex>[]_\phi & (C_*(X),d) \ar@<0.75ex>[r]^-p & (H_*(X),0). \ar@<0.75ex>[l]^-\iota }
		$$
Indeed, write $C_*(X)=A\oplus dA\oplus H$ with $H\cong H_*(X)$, choose $\iota$ to be the inclusion, and construct  $p$ and $\phi$ recursively so that they satisfy the required properties.

From now on,  whenever we refer to the $A_\infty$-structure on $H_*(X)$, this will be the one obtained via Theorem \ref{htt} applied to the above diagram.

We finish this section by exhibiting a class of spaces for which the cup product on $H^*(X)$ does not determine the $A_\infty$-structure on $H_*(X)$. For it, and until the end of the section  the coefficient field will be $\bq$ and we refer the reader to the Appendix which contains the notation and basic tools we will use from rational homotopy theory.

\begin{theorem}\label{otro}
Let $X$ be a non formal simply connected CW-complex of finite type and let $Y$ be the formal space associated to $H^*(X;\bq)$. Then, $X$ and $Y$ have isomorphic rational homology, isomorphic rational cohomology algebras but non isomorphic $A_\infty$-coalgebra structures.
\end{theorem}

\begin{proof}

Obviously, both spaces have isomorphic rational cohomology algebras, and in particular, isomorphic rational homology $H_*(X,\bq)\cong H_*(Y,\bq)$ which we will denote by $V$. By contradiction, assume that the $A_\infty$-coalgebra structures on $V$ induced by $X$ and $Y$ are isomorphic. This is equivalent to say that there is an isomorphism of DGA's
$$
(\widehat T(s^{-1}V), d_X)\cong  (\widehat T(s^{-1}V), d_Y)
$$
which restricts to an isomorphism of DGL's
$$
(\lib(s^{-1}V),\partial_X)\cong (\lib(s^{-1}V),\partial_Y)
$$
being these the Quillen minimal models of $X$ and $Y$ respectively. However, as $X$ is non formal and $Y$ is the formal space associated to its rational cohomology, $X$ and $Y$ cannot have the same homotopy type and thus, they do not have isomorphic Quillen models.
\end{proof}

\begin{example}
An explicit example given by the Theorem above is the following. Consider $X=(\bs^3\vee\bs^3)^{(5)}$ to be the $5^\text{\rm th}$ Postnikov stage of the wedge $\bs^3\vee\bs^3$. Its rational cohomology can be easily computed via Sullivan approach to rational homotopy theory to yield,
$$
H^*(X;\bq)\cong H^*(\Lambda (x_3,y_3,z_5),d),
$$
where: $\Lambda (x_3,y_3,z_5)$ denotes the free commutative (in the graded sense) algebra generated by two elements $x_3,y_3$ of degree $3$ and one element $z_5$ of degree $5$; the differential $d$ is zero on $x_3,y_3$ and $dz_5=x_3y_3$.

A straightforward computation show that this cohomology algebra has a basis (as graded vector space)
$$
\{ 1,\alpha_3,\beta_3,\gamma_{8},\rho_{8},\xi_{11}\}
$$
in which subscripts denote degree, and the only non trivial products are
$$
 \alpha_3\rho_{8}=-\beta_3\gamma_{8}=-\xi_{11}.
$$
A standard argument in rational homotopy theory shows that, if $Y$ is the formal space associated to this cohomology algebra, the rational homotopy of $Y$ is infinite dimensional contrary to the behaviour of $X$ and thus, they do not have the same rational homotopy type. This shows that $X$ is non-formal and the theorem above applies.

\end{example}

\section{$\Delta_n$-persistence and $\Delta_n$-barcodes}
From now on, and unless explicitly stated otherwise, any topological space considered will be of finite type over the field $\bk$, that is, its homology with coefficients in $\bk$ is finite dimensional in each degree.

Let ${\mathcal K}$ be a finite sequence of maps between topological spaces
	$$\xymatrix{
		K_0 \ar[r] & K_1 \ar[r] & \ldots \ar[r] & K_N.
	}$$
	For computational purposes, and in most of applied situations, just like \trd{as} in ordinary persistence, the above sequence will be a filtration of simplicial or cubical complexes, i.e., each of the maps are inclusion of subcomplexes of the complex $K_N$. However, from a theoretical point of view, no restriction is needed.

Denote by
	$$\xymatrix{
		H_*\left( K_i \right) \ar[r]^-{f^{i,j}} & H_*\left( K_j \right),
	}$$
the morphism induced in singular homology by the composition $K_i\to\dots\to K_j$.

We endow each $H_*(K_i)$ with the canonical structure of $A_\infty$-coalgebra denoted by $\{\De_n^i\}_{n\ge 1}$,  provided in Section 1. Moreover, each composition $K_i\to\dots\to K_j$ induces, via the DGC morphism $C_*(K_i)\to\dots\to C_*(K_j)$, a morphism of $A_\infty$-coalgebras ${\mathfrak f}^{i,j}\colon H_*(K_i)\to H_*(K_j)$ such that ${\mathfrak f}^{i,j}_{(1)}=f^{i,j}$. This data will be fixed along the remaining of this section. To avoid extra notation, we shall often write $\Delta_n$ instead of $\Delta_n^i$, as it will be always clear which term of the filtration we are considering.

\begin{definition} \label{Delta_nPersistentGroups}
		For every $0 \leq i \leq j \leq N$, and any $p\ge 0$, we define the {$p^{\text{\rm th}}$ $\De_n$-persistent group between $K_i$ and $K_j$} as
$$
(\De_n)_p^{i,j}\left( \mathcal K \right)=\im {f^{i,j}_p}|_{\cap_{k=i}^j \text{\rm Ker} \left( \De_n \circ f^{i,k}_p \right)},
$$
where $f^{i,j}_p\colon H_p(K_i)\to H_p(K_j)$.
We denote $(\De_n)^{i,j}\left( \mathcal K \right)=\oplus_{p\ge 0} (\De_n)_p^{i,j}\left( \mathcal K \right)$. Observe that we may also write,
$$ (\De_n)_p^{i,j}\left( \mathcal K \right)=
		\frac{\bigcap_{k=i}^j \text{\rm Ker} \left( \De_n \circ f^{i,k}_p \right)}
		{ \bigcap_{k=i}^j \text{\rm Ker} \left( \De_n \circ f^{i,k}_p \right) \cap \text{\rm Ker} f^{i,j}_p}. $$
	\end{definition}

This definition generalizes the classical persistence. In fact, if we denote by $H^{i,j}(\mathcal K)$ the classical $(i,j)$-persistence of $\mathcal K$, we have:

\begin{proposition} $(\De_1)^{i,j}\left( \mathcal K \right)=H^{i,j}(\mathcal K)$

\end{proposition}

\begin{proof} Note  that in each $H_*(K_i)$, $\De_1=0$. Hence,
$$
(\De_1)^{i,j}\left( \mathcal K \right)=H_*(K_i)/\ker f^{i,j}\cong \im  f^{i,j}=H^{i,j}(\mathcal K).
$$
\end{proof}

\begin{remark}\label{idea} To get a picture of what $\De_n$-persistence controls, broadly speaking, the $p^{\text{\rm th}}$ $\De_n$-persistent group between $K_i$ and $K_j$
	``captures''  the elements $\al \in H_p(K_i)$ such that
	$$\begin{matrix}
		H_p(K_i)& \longrightarrow & H_p(K_{i+1}) & \longrightarrow & \ldots & \longrightarrow & H_p(K_{j}) \\
		\al \neq 0 & & f^{i,i+1}\al \neq 0 & & \ldots & & f^{i,j} \al \neq 0 \\
		\De_n \al = 0 & & \De_n f^{i,i+1}\al = 0 & & \ldots & & \De_n f^{i,j} \al = 0,
	\end{matrix}$$
	modded out by the classes which have ``merged'' in $K_j$ or before. We say that $\al \in H_*(K_r)$ and $\be \in H_*(K_s)$ are {\em merged} in $K_k$  if
		$$f^{r,k} \al = f^{s,k} \be \in H_*(K_k).$$
	\end{remark}

When dealing with persistent homology over a field, its effective computation as well as its known geometrical meaning, are best understood through the {\em barcode} of the homology of the given filtration $\mathcal K$ \cite{CZ05,Ghr08}, which is obtained via the classification of finite graded modules over a principal ideal domain.

This process cannot be mimicked in $\De_n$-persistence,
the main obstruction being the following general fact: 
for a given $A_\infty$-coalgebra $(M, \{ \De_n \}_{n \geq 1} )$, and a given $n\ge 2$, the submodule $\ker \De_n$ is not preserved by $A_\infty$-morphisms, i.e., if $f\colon (M,\{\De_n\}_{n\ge 1})\to (N,\{\De'_n\}_{n\ge 1})$ is a morphism of $A_\infty$-coalgebras, $f_{(1)}(\ker \De_n)\nsubseteq \ker \De'_n$ in general.

In our particular case, this translate to the fact that, for any $0\le i\le N-1$, and any $n>2$ (as $\Delta_1=0$), in the general case,
$$
f^{i,i+1}\left( \ker \De_n^i\right) \not\subseteq \ker \De_n^{i+1}.
$$
Thus, we do not have an appropriate subsequence
\begin{equation}\label{quiensabe}
\xymatrix{
	\ker \De_n^0 \ar[r] & \ker \De_n^1 \ar[r] & \ldots \ar[r] &
	\ker \De_n^N ,
	}
\end{equation}
of the homology sequence of $\mathcal K$
$$\xymatrix{
		H_*\left( K_0 \right) \ar[r]^-{f^{0, 1}} &
		H_*\left( K_1 \right) \ar[r]^-{f^{1, 2}} & \ldots \ar[r]^-{f^{N-1,N}} &
		H_*\left( K_N \right),
	}$$
to which the standard barcode method of persistence homology could be applied.

Nevertheless, in the remaining of this section, we show how to assign to each filtration $\mathcal K$ a canonical code of  bars which effectively computes its $\De_n$-persistence.

For completeness, we start by  briefly recalling the concepts of {\em birth} and {\em death} of homology classes in classical persistence.

\begin{definition}\label{classical} Let $p\ge 0$ and $0\le i\le j\le N$. A homology class $\al \in H_p (K_i)$:
\begin{itemize}
			\item Is
				{\em Alive at $K_j$} if
				$f^{i,j} \al \neq 0$.
			\item Is {\em Born at $K_i$} if $\al$ is alive at $K_i$ (i.e., $\al \neq 0$)
			and $ \al \notin \im\, f^{i-1,i}$.
			\item	{\em Dies
				at $K_j$} if
				$\al$ is alive at $K_{j-1}$
				but not at $K_j$.
		\end{itemize}
\end{definition}

Next, we generalize the above to $\De_n$-persistence.

\begin{definition} Let $p\ge 0$ and $0\le i\le j\le N$. We say that a non zero class $\alpha\in H_p(K_i)$ is:
		
		\begin{itemize}
			\item {\em $\De_n$-awake at $K_j$} if $ f^{i,j} \al \neq 0 $ and $\De_n f^{i,j} \al = 0$.
			\item {\em $\De_n$-wakes up at $K_i$} if $\al$ is awake at $K_i$ (i.e., $\De_n\al=0$), 
			but it is not the image by $f^{i-1,i}$ of any awake class of $H_p(K_{i-1})$.	
			\item {\em $\De_n$-falls asleep at $K_j$} if $\al$ is $\De_n$-awake at $K_{j-1}$
				but not at $K_j$. In other words, if
				$$
				f^{i,j-1} \al \neq 0 \,\,\, \text{\rm and} \,\,\, \De_n f^{i,j-1} \al = 0,
				$$
				and either, $\al$ dies at $K_j$ (i.e.,  $f^{i,j} \al = 0$), or else, $ \De_n f^{i,j} \al \neq 0$.
		\end{itemize}
\end{definition}

\begin{remark}\label{observaciones} (i) Observe that, as $\De_1=0$ for any $H_*(K_i)$, the concepts of $\De_1$-awakening and $\De_1$-falling asleep coincide with those of birth and death in classical persistence.

(ii) It is important to point out, see Theorem \ref{nacermorir} below, that  in general, a non zero class $\al$ can $\De_n$-wake up and $\De_n$-fall asleep several times along $\mathcal K$.

\end{remark}

Nevertheless, the existence of a canonical barcode for $\Delta_n$-persistence will be easily deduced from the main result of this section which we now state.

\begin{theorem}\label{principal} Given $n\ge 1$, and $p\ge 0$, there exists a basis $\basis$  of $H(\mathcal K)=\oplus_{i=1}^NH_*(K_i)$  such that, for any $1\le i\le j\le N$,
$$
\dim (\Delta_n)^{i,j}_p(\mathcal K)=\#\{\beta\in \basis \cap H_p(K_i),\,\,\text{$\beta$ is $\Delta_n$-awake at $K_k$ for  $k=i,\dots,j$}\}.
$$
\end{theorem}

The rest of the section is devoted to the proof of this result. For that, to avoid extra notation,
as  $n$ and $p$ are fixed integers which do not depend on the methods of the proof, we will omit them and write $\Delta^i$ instead of the restriction of ${\Delta_n}$ to ${H_p(K_i)}$.

We begin by considering $\bv$

$$\xymatrix{
	V_0 & & V_{1} & & & \cdots & & & V_{N}, \\
	& W_{0} \ar[ul]^-{\iota} \ar[ur]_-{f^{0,1}} & &
	W_{1} \ar[ul]^-{\iota} \ar[ur]_-{f^{1,2}} & & & & W_{N-1} \ar[ul]^-{\iota} \ar[ur]_-{f^{N-1,N}}
	}$$
	to be the {\em zigzag module} \cite{CdS08} in which $\iota$ denotes inclusion and
	$$\begin{aligned}
		V_i &= \ker \De^i, \quad i=0, \ldots, N \\
		W_{i} &= \ker\Delta^i\cap\ker\Delta^{i+1}\circ f^{i,i+1},
		\quad  i=0, \ldots, N-1.
	\end{aligned}$$

For this kind of sequences there is a classification theorem which we now recall. In general, see \cite{CdS08} for a comprehensive treatment of this topic, a {\em zigzag module} $\bu$ is a sequence of vector spaces over the field $\bk$,
$$
U_1\stackrel{p_1}{\leftrightarrow}U_2\stackrel{p_2}{\leftrightarrow}\dots\stackrel{p_{r-1}}
{\leftrightarrow}U_r,
$$
in which each $p_i$ can be either a forward map $\stackrel{f_i}{\longrightarrow}$ or a backward map $\stackrel{g_i}{\longleftarrow}$.
A {\em submodule} $\bx\subset \bu$ of  $\bu$ is a zigzag module
 $$
X_1\stackrel{p_1}{\leftrightarrow}X_2\stackrel{p_2}{\leftrightarrow}\dots\stackrel{p_{r-1}}
{\leftrightarrow}X_r,
$$
defined by subspaces $X_i\subset U_i$, $i=1,\dots,r$, such that, for any $i$, either $f_i(X_i)\subset X_{i+1}$ or $g_i(X_{i+1})\subset X_i$. A submodule $\bx\subset \bu$ is a {\em direct summand}, or simply,  a {\em summand} of $\bu$, if there exists another submodule $\by\subset \bu$ such that $U_i=X_i\oplus Y_i$, for all $i$. We write $\bu=\bx\oplus \by$. It is important to remark that submodules are not, in general, summands. For instance \cite[Ex. 2.1]{CdS08}, in the zigzag module $\bu=\bk\stackrel{1}{\longrightarrow}\bk$, the map $\bx= 0\longrightarrow \bk$ constitutes a submodule but not a summand.

Given integers $1\le a\le b\le r$, an {\em interval $\,\,\bi[b,d]$ of birth time $b$ or that is born  at $I_b$, and death time $d$ or that dies at $I_b$}, is a  zigzag module
 $$
I_1\stackrel{p_1}{\leftrightarrow}I_2\stackrel{p_2}{\leftrightarrow}\dots\stackrel{p_{r-1}}
{\leftrightarrow}I_r,
$$
in which,
$$
I_i=\begin{cases} \bk,\quad &a\le i\le b,\\ 0,\quad&\text{otherwise},\end{cases}
$$
and
$$
p_i=\begin{cases} 1_{\bk},\quad &a\le i\le b-1,\\ 0,\quad&\text{otherwise}.\end{cases}
$$
Then, the Classification Theorem states that every zigzag module $\bu$ is isomorphic to a direct sum of intervals,
$$
\bu\cong\oplus_j\bi[a_j,b_j].
$$
For it, \cite{Gab72} is the original statement, \cite{DW05} contains a friendlier form, and finally, in \cite[4.1]{CdS08} the reader can find an  extended and constructive version of this statement which can be used  to develop explicit algorithms for the computation of the involved intervals.

It is important to remark that, given a zigzag module $\bu$,
the existence of a sequence $x_i\in U_i$, $i=1,\dots,n$, such that either $x_{i+1}=f_i(x_i)$ or $x_i=g_i(x_{i+1})$, whichever is applicable, does not produce an interval $\bi[1,n]$, see \cite[Caution 2.9 and Ex. 2.10]{CdS08}. In fact, our proof of Theorem \ref{principal} consists in showing that this is  the case for zigzag modules which behave as $\bv$.

	We begin by applying the Classification Theorem to obtain that $\bv$ is isomorphic 
	as zigzag module to a finite direct sum
	$$ \bi = \bigoplus_{m} \bi^m,$$
	 where each $\bi_m$ is an interval
	$$\xymatrix{
	I^m_0 & & I^m_{1} & & & \cdots & & & I^m_{N}, \\
	& J^m_{0} \ar[ul] \ar[ur] & &
	J^m_{1} \ar[ul] \ar[ur] & & & & J^m_{N-1} \ar[ul] \ar[ur]
	}$$
	of \trd{a ??} certain birth and death time. Notice that, from the isomorphism $\bv\cong\bi$, we can choose a basis  of each $\ker \De^i = V_i$ ,and then enlarge it to get a basis $\basis$  of the entire complex $H(\mathcal K)=\oplus_{i=1}^NH_*(K_i)$,
 so that
$$
\# \{ \beta\in \basis\cap H_p(K_i) \,|\, \beta \,\text{\rm is } \De^k-\text{\rm awake},\, k=i, \ldots, j\} =N_{ij},
$$
where $N_{ij}$ is the number
 of intervals $\bi^m$ whose birth  time is at most  at $I^m_i$, and its death time is at least at $I^m_j$.
	Hence, to prove Theorem \ref{principal},
  it suffices to show that,
	 $$ \dim (\De)^{i,j}\left( K \right)=N_{ij}.$$

To this end, we introduce some notation. Choose an explicit zigzag isomorphism,
	$$
	\Phi\colon \bv\stackrel{\cong}{\longrightarrow} \bi
	$$
which consists of $\bk$-linear isomorphisms,
	$$
\begin{aligned}
&\phi_k\colon V_k\stackrel\cong\longrightarrow \oplus_{m} I^m_k,\quad 0 \leq k \leq N,\\
&\varphi_k\colon W_k \stackrel\cong\longrightarrow \oplus_{m} J^m_k,\quad 0 \leq k \leq N-1,
\end{aligned}
$$
making commutative the corresponding zigzag squares:
\begin{equation}\label{diagramatrivial}
\xymatrix{
V_k\ar[d]^\cong_{\phi_k}&W_k\ar[r]^{\scriptscriptstyle{f^{k,k+1}}}\ar[d]_\cong^{\varphi_k}\ar@{_{(}->}[l]_\iota& V_{k+1}
\ar[d]_{\phi_{k+1}}^\cong\\
\oplus_{m}I^m_k &\oplus_{m}J^m_k\ar[r]\ar[l]&\oplus_{m}I^m_{k+1}.}
\end{equation}

We first show that $ \dim (\De)^{i,j}\left( K \right)\le N_{ij}$.
Keep in mind that a class $\alpha\in (\Delta)^{i,j}(\mathcal K)=\im {f^{i,j}}|_{\cap_{k=i}^j
\ker  \De^k \circ f^{i,k}} $ if and only if there exist elements,
$$
\alpha_\ell\in W_\ell\subset V_\ell,\quad\ell=i,\dots,j-1,
$$
such that,
$$
f^{\ell,\ell+1}(\alpha_{\ell})=\alpha_{\ell+1},\quad\ell=i,\dots,j-2,\quad\text{and}\quad f^{j-1,j}(\alpha_{j-1})=\alpha.
$$
In particular, given a non-zero $\alpha\in (\De)^{i,j}\left( K \right)$,  the $m^{\text{th}}$ coordinate of
 $\phi_j(\alpha)$, which we denote by $\phi_j^m(\alpha)$ henceforth, has to be non zero for some $m$, i.e.,
 $I^m_j\cong \bk$.
 As $\alpha=f^{j-1,j}(\alpha_{j-1})$, at the sight of diagram \ref{diagramatrivial} above, $\varphi_{j-1}^m(\alpha_{j-1})\not=0$ and thus, $J^m_{j-1}\cong \bk$ as well.

 On the other hand, as $V_k\stackrel\iota\hookleftarrow W_k$ are inclusions  and $\phi_k,\varphi_k$ are  isomorphisms,  also at the sight of  \ref{diagramatrivial} we have that $\phi_{j-1}^m(\alpha_{j-1})\not=0$, that is, $I^m_{j-1}\cong \bk$.

 Apply this process repeatedly to conclude that $I^m_i\cong\bk$ has birth  time  at most  at $I^m_i$, and its death time is at least at $I^m_j$.
 In particular, this shows that, while a given  interval $\bi^m$ may have death at some $J_{j}$, its birth is always at some $I_i$.

 Finally, note that any set of linearly independents elements of $ (\Delta)^{i,j}(\mathcal K)$ give raise to a set, of the same cardinal, of different intervals,
for it is enough to start the process above with a different $m$ for each of the elements of the linearly independent set. This shows the required inequality.

A similar argument shows that  $ N_{ij}\le \dim (\De)^{i,j}\left( K \right)$, and concludes the proof of Theorem \ref{principal}.

\begin{definition}  Let $I^m_{i_m}$ and $I^m_{j_m}$ or $J^m_{j_m}$ be the birth and death time respectively of the interval $\bi^m$ in the canonical decomposition   $\bv\cong\oplus_{m}\bi^m$. The {\em barcode associated to the $\Delta_n$-persistence of $\mathcal K$ in degree $p\ge 0$} is defined as the collection of bars, one for each $m$,  with origin at  $i_m$  and end at $j_m$.
\end{definition}

\begin{corollary} $ \dim (\De)^{i,j}\left( K \right)$ equals the number of bars in the barcode containing the interval $[i,j]$.\hfill$\square$
\end{corollary}

\section{Classes that wake up and fall asleep several times}

In this section we  show how, contrary to the case of the classical persistence environment, homology classes can wake up and fall asleep several times along the given filtration. We will assume the coefficient field to be $\bq$ and refer  the reader again to the Appendix for the basics of rational homotopy theory.

Let $\mathcal K$ be the  filtration of finite complexes
$$
K_0\stackrel{i_0}{\hookrightarrow}K_1\stackrel{i_1}{\hookrightarrow}K_2\stackrel{i_2}{\hookrightarrow}K_3
$$
defined as follows:

\begin{itemize}

 \item $K_0=(\bs_1^2\vee\bs_2^2\vee\bs_3^2\vee \bs^4)\cup_{g_1}e_1^6\cup_{g_2}e_2^6
 $
 in which $\bs_1^2,\bs_2^2,\bs_3^2$ simply denote three different copies of $\bs^2$ and the (homotopy classes of the) attaching maps for the $6$-cells $e_1^6, e_2^6$ are the following Whitehead products:
 $$
 g_1=[id_{\bs^4},id_{\bs_1^2}]+[id_{\bs_1^2},[id_{\bs_1^2},[id_{\bs_1^2},id_{\bs_2^2}]]],\quad g_2=[id_{\bs^4},id_{\bs_2^2}].
 $$

 \item $K_1=K_0\cup_{g_3}e^4$ in which
 $$
 g_3=[id_{\bs_1^2}, id_{\bs_2^2}].
 $$

 \item $
K_2=K_1\cup_{g_4}e^6
$
in which
$$
g_4=[id_{\bs^4}, id_{\bs_1^2}]-[id_{\bs_2^2},[id_{\bs_2^2},[id_{\bs_2^2},id_{\bs_3^2}]]].
 $$

 \item $
K_3=K_2\cup_{g_5}e^4
$
in which
$$
g_5=[id_{\bs_2^2}, id_{\bs_3^2}].
$$

\item The maps $i_1,i_2,i_3$ are the inclusions.
 \end{itemize}

Then, we prove,

\begin{theorem}\label{nacermorir} In
$$
H_*(K_0;\bq)\stackrel{f_0}{\longrightarrow}H_*(K_1;\bq)\stackrel{f_1}{\longrightarrow}H_*(K_2;\bq)\stackrel{f_2}{\longrightarrow}H_*(K_3;\bq),
$$
all the morphisms are injective,
and the non zero homology class $\gamma_6\in H_6(K_0)$, created by the cell $e_1^6$, is $\Delta_4$-asleep at $K_0$, it $\Delta_4$-wakes up at $K_1$, it is $\Delta_4$-asleep again at $K_2$ and finally, it $\Delta_4$-wakes up at $K_3$.
\end{theorem}

\begin{proof}

Via Theorem \ref{uno} of the Appendix,  the  Quillen minimal model of $K_0$ is  (from now on, subscripts will always denote degree),
$$
(\lib(V_0),\partial)=(\lib(x_1,x_1',x_1'',y_3,z_5,z_5'),\partial)
$$
where
$$
\begin{aligned}
&\partial x_1=\partial x_1'=\partial x''_1=\partial y_3=0,\\
&\partial z_5=[y_3,x_1]+[x_1,[x_1,[x_1,x_1']]],\\
&\partial z_5'=[y_3,x_1'].
\end{aligned}
$$

Observe, either by direct computation or 
in light
of the isomorphism (i) of the Appendix, that
$$
\widetilde H_*(K_0;\bq)=\langle \alpha_2,\alpha_2',\alpha_2'',\beta_4,\gamma_6,\gamma_6'\rangle=sV_0.
$$
Next, a short computation shows that, in the universal enveloping algebra $(T(V_0),d)=U(\lib(V_0),\partial)$ (see Appendix for its definition),
$$
d_4z_5=x_1\otimes x_1\otimes x'_1\otimes x_1+ x_1\otimes x_1\otimes x_1\otimes x'_1-x_1'\otimes x_1\otimes x_1\otimes x_1 - x_1\otimes x_1'\otimes x_1\otimes x_1.
$$
In other words, in the $A_\infty$-coalgebra structure on $ H_*(K_0;\bq)$,
$$
\Delta_4\gamma_6\not=0,
$$
and thus, $\gamma_6$ is $\Delta_4$-asleep.

 Next, again by  Theorem \ref{uno}, a Quillen model of the inclusion
$$
i_0\colon K_0\hookrightarrow K_1
$$
is
$$
(\lib(x_1,x_1',x_1'',y_3,z_5,z_5'), \partial)\hookrightarrow (\lib(x_1,x_1',x_1'',y_3,y_3',z_5,z_5'),\partial),
$$
in which
$$
\partial y_3'=[x_1,x_1'].
$$
To describe in simpler terms  the morphism induced in rational homology by the inclusion $i_0$, consider the isomorphism of (non differential!) free Lie algebras
$$
\psi\colon \lib(x_1,x_1',x_1'',y_3,y_3',u_5,z_5')\stackrel{\cong}{\longrightarrow} \lib(x_1,x_1',x_1'',y_3,y_3',z_5,z_5')
$$
which sends $u_5$ to $z_5-[x_1,[x_1,y_3']]$ and any other generator to itself.
Then, set in the first free Lie algebra the differential
$$
 \partial=\psi^{-1}\partial\psi
 $$
 so that $\psi$ becomes an isomorphism of DGL's. A short computation shows that the only generator on this isomorphic DGL in which the differential has changed is:
$$
\partial u_5=[y_3,x_1].
$$
To simplify the notation, set
$$
(\lib(V_1),\partial)=(\lib(x_1,x_1',x''_1,y_3,y_3',u_5,z_5'),\partial),
$$
and observe that the DGL morphism
$$
(\lib(V_0),\partial)\to(\lib(V_1),\partial)
$$
which sends $z_5$ to $u_5+[x_1,[x_1,y_3']]$, and any other generator to itself, is again the Quillen minimal model of $i_0\colon K_0\hookrightarrow K_1$. In particular, the map induced by this morphism on the suspension of the indecomposables $sV_0\to sV_1$ is precisely the morphism $f_0=H_*(i_0)$.

Thus, on the one hand,
$$
f_0\colon \widetilde H_*(K_0;\bq)\longrightarrow \widetilde H_*(K_1;\bq)
$$
is naturally identified to the inclusion
$$
sV_0=\langle \alpha_2,\alpha_2',\alpha_2'',\beta_4,\gamma_6,\gamma_6'\rangle\hookrightarrow
 \langle \alpha_2,\alpha'_2,\alpha''_2,\beta_4,\beta'_4,\eta_6,\gamma'_6\rangle=sV_1,
$$
for which  $f_0(\gamma_6)=\eta_6$.

On the other hand, one sees that in the  universal enveloping algebra $(T(V_1),d)=U(\lib(V_1),\partial)$,
$$
du_5=d_2 u_5=y_3\otimes x_1-x_1\otimes y_3.
$$
Equivalently, in the $A_\infty$-coalgebra structure on $H_*(K_1;\bq)$,
$$
\Delta_4\eta_6=\Delta_4 f_0(\gamma_6)=0.
$$
In other words, the class $\gamma_6$ $\Delta_4$-wakes up at $K_1$.

In the next stage,  Theorem \ref{uno} produces a Quillen model of the inclusion
$
i_1\colon K_1\hookrightarrow K_2
$,
$$
(\lib(x_1,x_1',x'',y_3,u_5,z_5'), \partial)\hookrightarrow (\lib(x_1,x_1',x'',y_3,y_3',u_5,z_5',v_5),\partial),
$$
in which
$$
\partial v_5=[y_3,x_1]-[x_1',[x'_1,[x'_1,x_1'']]].
$$
Once again, to describe in simpler terms  the morphism induced in rational homology by the inclusion $i_1$, consider the isomorphism of (non differential!) free Lie algebras
$$
\lib(x_1,x_1',x'',y_3,y_3',w_5,z_5',v_5)\stackrel{\cong}{\longrightarrow} \lib(x_1,x_1',x'',y_3,y_3',u_5,z_5',v_5)
$$
which sends $w_5$ to $u_5-v_5$ and any other generator to itself.
Then endow the left hand side free Lie algebra with the differential
 so that the above becomes an isomorphism of DGL's. A short computation shows that the only generator on this isomorphic DGL in which the differential has changed is:
$$
\partial w_5=[x_1',[x'_1,[x'_1,x''_1]]].
$$
Set
$$
(\lib(V_2),\partial)=(\lib(x_1,x_1',x''_1,y_3,y_3',w_5,z_5',v_5),\partial),
$$
and observe that the DGL morphism
$$
(\lib(V_1),\partial)\to(\lib(V_2),\partial)
$$
which sends $u_5$ to $w_5+v_5$, and any other generator to itself, is  the Quillen minimal model of $i_1\colon K_1\hookrightarrow K_2$. In particular the map induced by this morphism on the suspension of the indecomposables $sV_1\to sV_2$ is precisely $f_1=H_*(i_1)$.

Thus, on the one hand,
$$
f_1\colon \widetilde H_*(K_1;\bq)\longrightarrow \widetilde H_*(K_2;\bq)
$$
is naturally identified to the inclusion
$$
sV_1=\langle \alpha_2,\alpha'_2,\alpha''_2,\beta_4,\beta'_4,\eta_6,\gamma'_6\rangle\hookrightarrow\langle \alpha_2,\alpha'_2,\alpha''_2,\beta_4,\beta'_4,\rho_6,\gamma'_6,\gamma''_6\rangle=sV_2,
$$
for which  $f_1(\eta_6)=\rho_6$.

On the other hand, one sees that  in the universal enveloping algebra $(T(V_2),d)=U(\lib(V_2),\partial)$,
$$
dw_5=x'_1\otimes x'_1\otimes x''_1\otimes x'_1+x'_1\otimes x'_1\otimes x'_1\otimes x''_1-x''_1\otimes x'_1\otimes x'_1\otimes x'_1-x'_1\otimes x''_1\otimes x'_1\otimes x'_1.
$$
Equivalently, in the $A_\infty$-coalgebra structure on $H_*(K_1;\bq)$,
$$
\Delta_4\rho_6=\Delta_4 f^{0,1}(\gamma_6)\not=0.
$$
In other words, the class $\gamma_6$ $\Delta_4$-falls asleep at $K_2$.

Finally, we repeat this process once again, exhibiting via Theorem \ref{uno}, a Quillen model of the inclusion
$
i_2\colon K_2\hookrightarrow K_3
$,
$$
 (\lib(x_1,x_1',x'',y_3,y_3',w_5,z_5',v_5),\partial)\hookrightarrow (\lib(x_1,x_1',x'',y_3,y_3',y''_3,w_5,z_5',v_5),
$$
in which
$$
\partial y''_3=[x_1',x''_1].
$$
As in previous steps, we  get a DGL isomorphism
$$
(\lib(x_1,x_1',x'',y_3,y_3',y''_3,s_5,z_5',v_5),\partial)\stackrel{\cong}{\longrightarrow}( \lib(x_1,x_1',x'',y_3,y_3',y''_3,w_5,z_5',v_5),\partial),
$$
which sends $s_5$ to $w_5-[x'_1,[x'_1,y''_3]]$ and  is such that $\partial s_5=0$.
Set
$$
(\lib(V_3),\partial)=(\lib(x_1,x_1',x''_1,y_3,y_3',y''_3,s_5,z_5',v_5),\partial),
$$
and observe that the DGL morphism
$$
(\lib(V_2),\partial)\to(\lib(V_3),\partial)
$$
which sends $w_5$ to $s_5+[x'_1,[x'_1,y''_3]]$, and any other generator to itself, is  the Quillen minimal model of $i_2\colon K_2\hookrightarrow K_3$. In particular the map induced by this morphism  on the suspension of the indecomposables $sV_2\to sV_3$ is precisely the morphism $f_2=H_*(i_2)$.

Thus, on the one hand,
$$
f_2\colon \widetilde H_*(K_2;\bq)\longrightarrow \widetilde H_*(K_3;\bq)
$$
is naturally identified to the inclusion
$$
sV_2=\langle \alpha_2,\alpha'_2,\alpha''_2,\beta_4,\beta'_4,\rho_6,\gamma'_6,\gamma''_6\rangle \hookrightarrow
\langle \alpha_2,\alpha'_2,\alpha''_2,\beta_4,\beta'_4,\beta''_4,\varphi_6,\gamma'_6,\gamma''_6\rangle=sV_3,
$$
for which  $f_2(\rho_6)=\varphi_6$.

On the other hand, in the in the universal enveloping algebra $(T(V_3),d)=U(\lib(V_3),\partial)$,
$$
ds_5=0.
$$
Equivalently, in the $A_\infty$-coalgebra structure on $H_*(K_3;\bq)$,
$$
\Delta_4\varphi_6=\Delta_4 f^{0,2}(\gamma_6)=0,
$$
and the class $\gamma_6$ $\Delta_4$-wakes up again at $K_3$.

\end{proof}

\section{Appendix: basics of rational homotopy theory}
\label{RHT}
In this appendix, for completeness, we recall the basic facts and classical results from rational homotopy theory used along the paper.
For a compendium of this theory, we direct the reader to the standard reference   \cite{FHT00}.

A {\em differential graded Lie algebra} (DGL henceforth) is a differential graded vector space $(L,\partial)$ in which:

\begin{itemize}

\item[$\bullet$] $L=\oplus_{p\in\bz}L_p$ in endowed with a linear operation, called Lie bracket,
$$
[\,\,,\,]\colon L_p\otimes L_q\longrightarrow L_{p+q},\qquad p,q\in \bz,
$$
 satisfying {\em antisymmetry},
$$
[x,y]=(-1)^{|x||y|+1}[y,x],
$$
and {\em Jacobi identity},
$$
\bigl[x,[y,z]\bigr]=\bigl[[x,y],z\bigr]+(-1)^{|x||y|}\bigl[y,[x,z]\bigr],
$$
for any  homogeneous elements $x,y,z\in L$. In other words, $L$ is a graded Lie algebra.

\item[$\bullet$] The differential $\partial$, of degree $-1$, satisfies the {\em Leibniz rule},
$$
\partial[x,y]=[\partial x,y]+(-1)^{|x|}[x,\partial y],
$$
for any  pair of homogeneous elements $x,y\in L$.
\end{itemize}

The tensor algebra
$T(V)=\oplus_{n\ge 0} T^n(V)$ generated by the graded vector space $V$ is endowed with a graded Lie algebra structure with the brackets given by commutators:
$$
[a,b]=a\otimes b -(-1)^{|a||b|}b\otimes a
$$
for any homogeneous elements $a,b\in T(V)$.  Then, the {\em free Lie algebra} $\lib(V)$ is the Lie subalgebra of $T(V)$ generated by $V$.

Observe that $\lib (V)$ is filtered as follows,
$$
\lib(V)= \oplus_{n\ge 1}\lib^n(V),
$$
in which $\lib^n(V)$ is the vector space spanned by  Lie brackets of length $n$, that is $\lib^n(V)=\lib(V)\cap T^n(V)$. In particular, to set a differential in $\lib(V)$,
it is enough to define linear maps,
$$
\partial_n\colon V\longrightarrow \lib^n(V),\quad n\ge 1
$$
in such a way that $\partial=\sum_{n\ge 1}\partial_n$ squares to zero  and it satisfies the Leibniz rule. Note that $\partial_1$, the so called {\em
linear part of} $\partial$, is therefore a differential in $V$.

\begin{remark}\label{remark}

It is important to remark that, if $T(V)$ is endowed with a derivation $d$ for which $(T(V),d)$ is a  differential graded (non commutative) algebra, $\lib(V)$ does not inherit a differential as $d$ may not respect commutators. However, every time we have a DGL of the form $(\lib(V),\partial)$, the differential $\partial$ can be extended as a derivation of graded algebras to $T(V)$ to make it a differential graded algebra. In other words, consider the functor
$$
U\colon \catdgl\longrightarrow \catdga
$$
 which associates to every DGL $(L,\partial)$ its {\em universal enveloping algebra} $UL$ which is the graded algebra
 $$
 T(L)/\langle [x,y]-(x\otimes y -(-1)^{|a||b|}y\otimes x)\rangle,\quad x,y\in L,
 $$
 with the differential induced by $\partial$. Then, whenever $(L,\partial)=(\lib(V),\partial)$, one has,
 $$
 U (\lib(V),\partial)=(T(V),\partial).
 $$

 \end{remark}

In \cite{Qui69}, D. Quillen associates to every $1$-connected topological space $X$ of the homotopy type of a CW-complex, a particular differential graded Lie algebra  $(\lib(V),\partial)$ which is {\em reduced}, i.e., $V=\oplus_{p\ge 1} V_p$ for which:
\begin{itemize}

\item[(i)] $H_*(V,\partial_1)\cong s^{-1}\widetilde H_*(X;\bq)=\widetilde H_{*-1}(X;\bq)$.

\item[(ii)] $H_*(\lib(V),\partial)\cong \pi_*(\Omega X)\otimes \bq\cong \pi_{*+1}(X)\otimes\bq$.
\end{itemize}

This is a {\em Quillen model of $X$}  and it is called {\em minimal}  whenever $\partial$ is decomposable, i.e., $\partial_1=0$. In this case, (i) becomes,
$$
V\cong s^{-1}\widetilde H_*(X;\bq).
$$

The minimal Quillen model is unique up to isomorphism. Moreover, its  construction is functorial and it defines an equivalence between the homotopy category of $1$-connected spaces of the homotopy type of rational CW-complexes and that of reduced differential graded Lie algebras over $\bq$. Thus, two simply connected complexes have isomorphic Quillen minimal models if and only if they have the same rational homotopy type.

In particular, if $\varphi\colon (\lib(U),\partial_U)\to (\lib(V),\partial_V)$ is the Quillen minimal model of the map $f\colon X\to Y$, the morphism $H_*(\varphi)\colon H_*(\lib(U),\partial_U)\to H_*(\lib(V),\partial_V)$ is naturally identified to $\pi_*(\Omega f)\colon \pi_*(\Omega X)\to \pi_*(\Omega Y)$. In the same way, the morphism induced by $\varphi$ at the ``indecomposables'' $\varphi_0\colon U\to V$ is naturally identified to $s^{-1}H_*(f)\colon s^{-1}\widetilde H_*(X;\bq)\to s^{-1}\widetilde H_*(Y;\bq)$, the desuspension of the morphism induced by $f$ in rational homology. To explicitly obtain $\varphi_0$, write, for each $u\in U$, $\varphi(u)=v+\Gamma$, with $v\in V$ and  $\Gamma\in \lib^{\ge 2}(V)$. Then $\varphi_0(u)=v$.

In the $A_\infty$-language, Remark \ref{remark} together with the isomorphism (i) above, asserts that $\widetilde H_*(X;\bq)$ can be functorially endowed with a structure of $A_{\infty}$-coalgebra which is precisely the standard structure defined in \S1.

Given a $1$-connected CW-complex $X$, a (non necessarily minimal)  Quillen model of $X$ can be described in terms of  a CW-decomposition of $X$ via the following,

\begin{theorem}\label{uno}\cite[III.3]{Tan83} Let $Y=X\cup_fe^{n+1}$ be the space obtained attaching  an $n+1$-cell to $X$ via the map $f\colon \bs^n\to X$. Let $(\lib(V),\partial)$ be a Quillen model of $X$ and let $\Phi\in \lib(V)_{n-1}$ be a cycle representing the homology class in $H_{n-1}(\lib(V),\partial)$ which is identified, via the isomorphism (ii) above, with the homotopy class $[f]\in\pi_n(X)\otimes\bq$. Then, the injection
$$
(\lib(V),\partial)\hookrightarrow (\lib(V\oplus \bq a),\partial),\qquad \partial a=\Phi
$$
is a Quillen Model of the inclusion
$$
X\hookrightarrow Y.
$$
\end{theorem}

Finally, we recall that a simply connected complex is {\em formal} if its rational homotopy type depends only on its rational cohomology algebra. Given a commutative graded algebra $H$ there is a formal simply connected complex $X$, unique up to rational homotopy, such that $H^*(X;\bq)=H$.

\vspace{10mm}
\noindent \textsc{Departamento de \'Algebra, Geometr\'ia y Topolog\'ia, 
Universidad de M\'alaga, Ap. 59, 29080 M\'alaga, Spain.}
\\
frbegu@gmail.com
\\
aniceto@uma.es

\end{document}